\documentclass[11pt, twoside]{amsart}
\usepackage[cp1251]{inputenc}
\usepackage{color}
\usepackage{amsthm,amssymb,latexsym}
\usepackage[centertags]{amsmath}
\usepackage[T2A]{fontenc}
\usepackage{MnSymbol}
\usepackage{bbm}
\usepackage{centernot}
\usepackage{mathrsfs}
\usepackage{amsthm}
\usepackage{amsfonts}
\usepackage{textcomp}
\usepackage{newlfont}
\usepackage{enumitem}

\newif\ifPDF
\ifx\pdfoutput\undefined\PDFfalse
\else \ifnum \pdfoutput > 0 \PDFtrue
        \else \PDFfalse
        \fi
\fi

\ifPDF
  \usepackage[pdftex]{xcolor, graphicx}
  \usepackage[colorlinks=true,linkcolor=blue, citecolor=blue]{hyperref}%

\else
  \usepackage{color}
  \usepackage[dvips]{graphicx}
  \usepackage[dvips]{hyperref}
\fi

\usepackage{tkz-graph}
\tikzset{EdgeStyle/.style = {->}}
\tikzset{LabelStyle/.style= {fill=yellow}}
\usetikzlibrary{shapes,snakes,calendar,matrix,backgrounds,folding}

\usepackage{bbm}

\usepackage[top=1in, bottom=1.25in, left=1.25in, right=1.25in]
{geometry}

\theoremstyle{definition}
\newtheorem{definition}{Definition}[section]
\theoremstyle{remark}

\newtheorem{remark}[definition]{Remark}

\theoremstyle{plain}
\newtheorem{thm}[definition]{Theorem}
\newtheorem{theorem}[definition]{Theorem}
\newtheorem{prop}[definition]{Proposition}
\newtheorem{lemma}[definition]{Lemma}

\newtheorem{corol}[definition]{Corollary}

\def\ol{\overline}
\def\ov{\overline}
\def\tl{\widetilde}
\def\wh{\widehat}

\newcommand{\be}{\begin{equation}}
\newcommand{\ee}{\end{equation}}
\newcommand{\ba}{\begin{aligned}}
\newcommand{\ea}{\end{aligned}}

\numberwithin{equation}{section}
\newcommand{\ignore}[1]{}

\newcommand{\ie}{\emph{i.e.}}

\begin{document}

\title[Measures and dynamics on Pascal-Bratteli diagrams]{Measures and dynamics on Pascal-Bratteli diagrams}


\author{Sergey Bezuglyi}
\address{Department of Mathematics,
University of Iowa, Iowa City, IA 52242-1419
USA}
\email{sergii-bezuglyi@uiowa.edu}

\author{Artem Dudko}
\address{Institute of Mathematics of the Polish Academy of Sciences, ul. \'Sniadeckich 8, 00-656 Warsaw, Poland \&
B. Verkin Institute for Low Temperature Physics and Engineering,
47~Nauky Ave., Kharkiv, 61103, Ukraine}

\email{adudko@impan.pl}

\author{Olena Karpel}
\address{AGH University of Krakow, Faculty of Applied Mathematics, al. Adama Mickiewicza~30, 30-059 Krakow, Poland \&
B. Verkin Institute for Low Temperature Physics and Engineering,
47~Nauky Ave., Kharkiv, 61103, Ukraine}

\email{okarpel@agh.edu.pl}

\begin{abstract}
We introduce and study dynamical systems and measures on stationary generalized Bratteli diagrams $B$ that are 
represented as the union of countably many classical 
Pascal-Bratteli diagrams. We describe all ergodic tail invariant measures on $B$. For every probability tail invariant measure $\nu_p$ on the classical Pascal-Bratteli diagram, we approximate 
the support 
of $\nu_p$ by the path space of a subdiagram. By considering various orders on the edges of $B$,
we define dynamical systems with various properties. 
We show that there exist orders such that the sets of infinite maximal and infinite minimal paths are empty. This implies that the corresponding Vershik map is a homeomorphism.
We also describe orders on both $B$ and the classical Pascal-Bratteli diagram that generate either uncountably many minimal infinite and uncountably many maximal infinite paths, or uncountably many minimal infinite paths alongside countably infinitely many maximal infinite paths.
 
\end{abstract}

\maketitle








\section{Introduction}

We have initiated the study of generalized Bratteli diagrams in a series of recent works (see, e.g., \cite{BezuglyiDooleyKwiatkowski2006, BezuglyiJorgensenSanadhya2024, BezuglyiJorgensenKarpelSanadhya2023, BezuglyiKarpelKwiatkowski2024, BezuglyiKarpelKwiatkowskiWata2024}). In this paper, we focus on investigating invariant measures and dynamics for both the standard Pascal-Bratteli diagram and a broader class of generalized Bratteli diagrams, which can be represented as countable unions of classical Pascal graphs.

Generalized Bratteli diagrams serve as models for Borel automorphisms of standard Borel spaces. The tail equivalence relation on the path space of a generalized Bratteli diagram describes the dynamical properties of the corresponding automorphisms. These diagrams feature countably infinite vertices at each level, and their path spaces are non-compact Polish spaces. The corresponding incidence matrices are also infinite. These characteristics introduce new challenges in studying invariant measures and dynamics for generalized Bratteli diagrams, compared to the standard Bratteli diagrams typically used in Cantor dynamics. 

The set of various classes of generalized Bratteli diagrams is large and significantly different from that of standard Bratteli diagrams. For example, there are no simple generalized Bratteli diagrams, although the tail equivalence relation can be minimal for some diagrams.

The class of stationary standard Bratteli diagrams plays an important role in this area, as they provide models for substitution dynamical systems on finite alphabets \cite{Forrest1997, Durand_Host_skau_1999, BezuglyiKwiatkowskiMedynets2009}. The set of invariant measures for stationary standard Bratteli diagrams is fully described in \cite{BezuglyiKwiatkowskiMedynetsSolomyak2010}. However, for stationary generalized Bratteli diagrams, the situation is much more complex. The results in \cite{BezuglyiKwiatkowskiMedynetsSolomyak2010} no longer apply because generalized Bratteli diagrams can admit infinite $\sigma$-finite measures that assign finite values to all cylinder sets (see \cite{BezuglyiKarpelKwiatkowski2024}). It was shown in \cite{BezuglyiJorgensenSanadhya2024} that stationary generalized Bratteli diagrams serve as models for a class of substitution dynamical systems on infinite alphabets. Substitutions on infinite alphabets have also been studied in works such as \cite{Ferenczi2006, Manibo_Rust_Walton_2022}, and \cite{Frettloh_Garber_Manibo_2022}

Pascal-Bratteli diagrams (also called Pascal diagrams, Pascal graphs, Pascal adic systems) have been studied in numerous papers, see, for example, \cite{Boca2008, FrickOrmes2013, FrickPetersen2010, Fricketal2017, MelaPetersen2005, Mela2006, Mundici2011, Strasser2011, Vershik2014}. In \cite{BezuglyiKarpelKwiatkowskiWata2024}, $\mathbb{N}$-infinite and $\mathbb{Z}$-infinite generalized Pascal-Bratteli diagrams were studied, and the set of probability ergodic invariant measures for these diagrams was completely described.

In this paper, we study ergodic invariant probability measures on standard Pascal-Bratteli diagrams using vertex subdiagrams. We describe subdiagrams of the standard Pascal-Bratteli diagram that are pairwise disjoint for different ergodic invariant probability measures and can be chosen so that the corresponding ergodic measure of their path space is arbitrarily close to 1.

We introduce stationary generalized Bratteli diagrams $B$, which are represented as the union of countably many standard Pascal-Bratteli diagrams, and describe all ergodic invariant probability measures on $B$. Additionally, we study various orders on both the standard Pascal-Bratteli diagram and the generalized Bratteli diagram $B$, which contains infinitely many classical Pascal-Bratteli diagrams. We introduce an order on the standard Pascal-Bratteli diagram that has uncountably many minimal infinite and uncountably many maximal infinite paths, and an order that has uncountably many minimal infinite and countably infinitely many maximal infinite paths. Conversely, we show that $B$ can be endowed with an order that has no minimal infinite or maximal infinite paths (in which case the corresponding Vershik map is a homeomorphism), as well as with an order that has uncountably many minimal infinite and uncountably many maximal infinite paths and an order that has uncountably many minimal infinite and countably infinitely many maximal infinite paths.

The outline of the paper is as follows.
In Section~\ref{Sect:prelim}, we provide all necessary definitions and recall the procedure for measure extension from a subdiagram.
 In Section \ref{Sect:pos_meas_sbd_standard_Pascal}, we describe subdiagrams $X_p$ of a standard Pascal-Bratteli diagram such that $\{X_p\}_{p\in (0,1)}$ are pairwise disjoint for different ergodic probability invariant measures $\nu_p$, and can be chosen so that the measure $\nu_p (X_p)$ is arbitrary close to $1$. 
Section \ref{Sect:X_max_X_min} is devoted to various orders on the standard Pascal-Bratteli diagram. We construct an order on the standard Pascal-Bratteli diagram that has a continuum of minimal infinite paths and a continuum of maximal infinite paths. We also answer the question posed in \cite{Fricketal2017} and construct an order on the standard Pascal-Bratteli diagram that has a continuum of minimal infinite paths and countably infinitely many maximal infinite paths. 
In Section~\ref{Sect:meas_GBD_Pascal},
we consider a stationary generalized Bratteli diagram $B$ which contains countably many standard Pascal-Bratteli diagrams as vertex subdiagrams.
We describe all probability ergodic invariant measures on $B$. Using the results from Section \ref{Sect:X_max_X_min}, we show that $B$ can be ordered so that it has
 a continuum of minimal infinite paths and a continuum of maximal infinite paths, thus providing an example of a stationary generalized Bratteli diagram which has such a property. We also construct an order on $B$ that has uncountably many minimal infinite and countably infinitely many maximal infinite paths. We also demonstrate that there is an order on the two-sided version of $B$ that has no maximal infinite paths and no minimal infinite paths; thus, the corresponding Vershik map is a homeomorphism. 

Our main results are contained in Theorems \ref{thm: measure supports}, \ref{thm: continuum Xmin}, \ref{thm: continuum Xmin countable Xmax}, \ref{Thm:all_mu_on_GBD_Pascal}.

\section{Preliminaries}\label{Sect:prelim} 
In this section, we briefly remind the reader of the definitions  
of main objects considered in the paper. 

\subsection{Standard and generalized Bratteli diagrams}

Standard Bratteli diagrams and Vershik maps on them are models for 
homeomorphisms of a Cantor set \cite{HPS1992, Medynets2006, 
DownarowiczKarpel_2019}. 

\begin{definition}
A {\it standard Bratteli diagram} is an infinite graded graph $B=(V,E)$ such that the vertex
set $V =\bigsqcup_{i\geq 0}V_i$ and the edge set $E=\bigsqcup_{i\geq 0}E_i$
are partitioned into disjoint subsets $V_i$ and $E_i$, where

(i) $V_0=\{v_0\}$ is a single point;

(ii) $V_i$ and $E_i$ are finite sets for all $i$;

(iii) there exists a range map $r \colon E \rightarrow V$ and a source map $s \colon E \rightarrow V$ such that $r(E_i)= V_{i+1}$ and $s(E_i)= V_{i}$ for all $i \geq 1$.
\end{definition}

A generalized Bratteli diagram is a natural extension of the notion of a standard Bratteli diagram. Generalized Bratteli diagrams have countably many vertices on each level and provide models for Borel automorphisms of standard Borel spaces, \cite{BezuglyiDooleyKwiatkowski2006,
BezuglyiJorgensenSanadhya2024}.

\begin{definition}\label{Def:generalized_BD} A 
\textit{generalized Bratteli diagram} is an infinite graded graph 
$B = (V, E)$ such that the 
vertex set $V$ and the edge set $E$ can be partitioned $V = \bigsqcup_{i=0}^\infty  V_i$ and $E = 
\bigsqcup_{i=0}^\infty  E_i$ so that the following properties hold (there is no need to assume that $V_0$ is a singleton): 

(i) For every $i \in \mathbb{Z}_+$, the number of vertices at each level 
$V_i$ is countably infinite, and the set $E_i$
of all edges between $V_i$ and $V_{i+1}$ is countable.

(ii) For every edge $e\in E$, we define the 
range and 
source maps $r$ and $s$ such that $r(E_i) = V_{i+1}$ and 
$s(E_i) = V_{i}$ for $i \in \mathbb{Z}_+$.

(iii) For every vertex $v \in V \setminus V_0$, we 
have $|r^{-1}(v)| < \infty$.  
\end{definition} 

A (finite or infinite) \textit{path} in the diagram is a (finite or infinite) sequence of edges $(e_i: e_i\in E_i)$ such that $s(e_i)=r(e_{i-1})$. Denote by $X_B$ the set of all infinite paths that start at $V_0$. The set 
$$
    [\ol e] := \{x = (x_i) \in X_B : x_0 = e_0, ..., x_n = e_n\}, 
$$ 
is called the \textit{cylinder set} associated with a finite path $\ol e = (e_0, ... , e_n)$. Cylinder sets generate a topology on $X_B$ such that
$X_B$ becomes a 0-dimensional Polish space. Recall that, for a standard Bratteli diagram, $X_B$ is compact, while the path space of a generalized Bratteli diagram is just a Polish space (it can be locally compact for 
some special classes of generalized Bratteli diagrams).

For $n = 0,1,\ldots$, let the element $f^{(n)}_{v,w}$ of the $n$-th incidence matrix $F_n = (f^{(n)}_{v,w})$ be the number of all edges between $v \in V_{n+1}$ and $w \in V_n$. If $F_n = F$ for all $n$, we call the corresponding generalized Bratteli diagram stationary. If all incidence matrices $F_n$ are $\mathbb{N}\times \mathbb{N}$ matrices, we call the corresponding generalized Bratteli diagram $B$ one-sided. If all $F_n$ are $\mathbb{Z}\times \mathbb{Z}$ matrices, we call $B$ two-sided.

\subsection{Vershik map and tail equivalence relation}

To define dynamics on Bratteli diagrams, we need the notion of an ordered Bratteli diagram.
An \textit{ordered} (standard or generalized) Bratteli diagram $B = (B, V, >)$ is a (standard or generalized) Bratteli diagram $B=(V,E)$ together with a partial order $>$ on $E$ such that edges $e,e'$ are comparable if and only if $r(e)=r(e')$ (see \cite{HPS1992, BezuglyiKarpel2016, BezuglyiJorgensenKarpelSanadhya2023} for more 
details). A (finite or infinite) path $e= (e_0,e_1,..., 
e_i,...)$ is called \textit{maximal} (respectively 
\textit{minimal}) if every $e_i$ is 
maximal (respectively minimal) among all elements of $r^{-1}(r(e_i))$. We denote the sets of all 
maximal infinite and all minimal infinite paths by $X_{\max}$ and 
$X_{\min}$ respectively. It is not hard to see that these sets are closed.
For brevity, we call the elements of $X_{\max}$ and 
$X_{\min}$ extreme paths. Note that, for a standard  Bratteli diagram, the sets $X_{\max}$ and 
$X_{\min}$ are necessarily nonempty, while for generalized Bratteli diagrams, there are orders for which $X_{\min} = X_{\max} = \emptyset$ (see Section 4 in \cite{BezuglyiJorgensenKarpelSanadhya2023} or \cite{BezuglyiDooleyKwiatkowski2006}). 

To introduce the Vershik map, first define a Borel transformation   $\varphi_B : X_B \setminus X_{\max} \rightarrow X_B \setminus X_{\min}$ as 
follows: given $x = (x_0, x_1,...)\in 
X_B\setminus X_{\max}$, let $k$ be the smallest number such that $x_k$ is not maximal. Let $y_k$ be the successor of $x_k$ in the finite set $r^{-1}(r(x_k))$. Define $\varphi_B(x):= (y_0, y_1,...,y_{k-1}, y_k,x_{k+1},...)$
where $(y_0, ..., y_{k-1})$ is the unique minimal path from 
$s(y_{k})$ to $V_0$.  Note that such defined $\varphi_B$  is, in fact, a homeomorphism. To extend the definition of $\varphi_B$ to $X_B$, we need to determine a bijection from $X_{\max} \rightarrow  X_{\min}$. It can be always done in a Borel manner if $X_{\max}$ and $X_{\min}$ have the same cardinality. Then $(X_B, \varphi_B)$ is called a generalized Bratteli-Vershik dynamical system associated with an ordered Bratteli diagram $(B, >)$.
In some cases (for example when $X_{\max}$ and $X_{\min}$ are empty sets), this extension can be made continuous, and $\varphi_B$ will be a 
homeomorphism of $X_B$. These questions have been discussed in 
\cite{BezuglyiJorgensenKarpelSanadhya2023, BezuglyiKarpelKwiatkowski2024, BezuglyiDooleyKwiatkowski2006}).

Two paths $x= (x_i)$ and $y=(y_i)$ in $X_B$ are called 
\textit{tail equivalent} if there exists an $n \in \mathbb{Z}_+$ 
such that $x_i = y_i$ for all $i \geq n$. This notion defines a 
countable Borel equivalence relation $\mathcal R$ on the
path space $X_B$ which is called the \textit{tail equivalence 
relation}.

Throughout the paper, by the term \textit{measure} we will mean a
non-atomic positive Borel measure on a Polish space. We will use the 
fact that any such measure is completely determined by its values on cylinder sets. 

\begin{definition}\label{def: tail inv meas}  A measure $\mu$ on 
$X_B$ is called \textit{tail invariant} if, for any cylinder sets
$[\ol e]$ and $[\ol e']$ such that $r(\ol e) = r(\ol e')$, we have
$\mu([\ol e]) = \mu([\ol e'])$.
\end{definition}


\subsection{Measure extension from subdiagrams}

In this subsection, we briefly describe the procedure of measure extension from a vertex subdiagram of a Bratteli diagram. This procedure works in the same way for both standard and generalized Bratteli diagrams. For more details see \cite{BezuglyiKarpelKwiatkowski2015, AdamskaBezuglyiKarpelKwiatkowski2017, BezuglyiKarpelKwiatkowski2024}.

Let $B$ be a standard or generalized Bratteli diagram.
A vertex subdiagram $\ov B$ of $B$ is a (standard or generalized) Bratteli diagram such that the set of vertices $W$ is formed by nonempty proper subsets $W_n \subset V_n$ and the set of edges consists of all edges of $B$ whose source and range are in $W_{n}$ and $W_{n+1}$, respectively. In other words, the incidence matrix $\ol F_n$ of $\ol B$ has the size $|W_{n+1}| \times |W_n|$ (it may be $\infty \times \infty$), and it is represented by a block of $F_n$ corresponding to the vertices from $W_{n+1}$ and $W_{n}$.
Note that for a generalized Bratteli diagram $B$, we consider both standard and generalized Bratteli subdiagrams $\ol B$, i.e., $|W_n|$
can be finite or infinite.

Let $B$ be a standard or generalized Bratteli diagram and $\ov B$ be its vertex subdiagram. Let $\ov \mu$ be an ergodic tail invariant probability measure on $X_{\ol B}$. Then $\ol \mu$ can be canonically extended to the ergodic measure $\widehat{\ov \mu}$ on the space $X_{B}$ by tail invariance: let $p_w^{(n)}$ be a measure $\ov \mu([\ov e])$ of any cylinder set $[\ov e] \subset X_{\ov B}$ such that $r(\ov e) = w \in W_n$. Then for any cylinder set $[\ov e'] \subset X_{B}$ such that $r(\ov e') = w \in W_n$ we set $\wh {\ov \mu}([\ov e']) := {\ov \mu}([\ov e]) = \ov p_w^{(n)}$. Denote by $\widehat X_{\ol B}:= \mathcal R(X_{\ol B})$ the subset of all paths in $X_B$ that are tail equivalent to paths from $X_{\ol B}$. Then $\widehat X_{\ol B}$ is the smallest $\mathcal R$-invariant subset of $X_B$ containing $X_{\ol B}$. After setting $\wh {\ov \mu}(X_B \setminus \wh X_{\ov B}) = 0$, we obtain an ergodic tail invariant measure $\wh {\ov \mu}$ on the whole path space $X_B$.

One can compute the measure $\wh {\ov \mu}(X_B)$ as follows. Let $\wh X_{\ol B}^{(n)}$ be the set of all paths $x = (x_i)_{i = 0}^{\infty}$ from $X_B$ such that the finite path  $(x_0, \ldots, x_{n-1})$ ends at a vertex $w \in W_n$ of $\ov B$, and the tail  $(x_{n},x_{n+1},\ldots)$ belongs to $\overline{B}$,
i.e., 
\begin{equation}\label{n-th level}
\wh X_{\ol B}^{(n)} = \{x = (x_i)\in \wh X_{\ol B} : r(x_{i-1}) \in W_{i}, \ \forall i \geq n\}.
\end{equation}
It is clear that $\wh X_{\ol B}^{(n)} \subset \wh X_{\ol B}^{(n+1)}$ and $\wh X_{\ol B} = \bigcup_n \wh X_{\ol B}^{(n)}$. Moreover, we have
\begin{equation}\label{extension_method}
\widehat{\ov \mu}(\wh X_{\ol B}) = \lim_{n\to\infty} \widehat{\ov \mu}(\wh X_{\ol B}^{(n)}) = 
\lim_{n\to\infty}\sum_{w\in W_n}  H^{(n)}_w \ov p^{(n)}_w =  1 + \sum_{n = 0}^{\infty}\sum_{v \in W_{n + 1}}
\sum_{w\in {W}'_{n}} {f}_{vw}^{(n)}  H_{w}^{(n)} 
\ol p_v^{(n+1)},
\end{equation}
where $W_{n}^{'} = V_{n} \setminus W_{n},\ n = 0, 1, 2, \ldots$ and $H_w^{(n)}$ is a number of paths between the vertex $w \in V_n$ and vertices of $V_0$ (see Theorem 3.1 in \cite{BezuglyiKarpelKwiatkowski2024}). The value 
$\wh {\ov \mu}(X_B)$ can be either finite or infinite. If it is finite, then we say that $\ol\mu$ admits a finite measure extension.

\section{Disjoint subdiagrams for ergodic invariant measures}\label{Sect:pos_meas_sbd_standard_Pascal} 

In \cite[Subection 5.2]{BezuglyiKarpelKwiatkowski2019}, a class of (standard) Bratteli diagrams was considered such that every ergodic invariant probability measure can be obtained as the measure extension from a uniquely ergodic vertex subdiagram. It was shown in \cite[Subsection 6.2]{BezuglyiKarpelKwiatkowski2019} that the Pascal-Bratteli diagram does not belong to this class. In this section, we prove that, for every ergodic invariant probability measure $\nu_p$ on the Pascal-Bratteli diagram, one can find a subdiagram $B_p$ such that the measure $\nu_p(X_p) $ is arbitrarily
close to 1, and the sets $X_p$ are pairwise disjoint, where $X_p$ is the path space of $B_p$, $0<p<1$. 

\begin{figure}[ht]
\centerline{\includegraphics[width=0.5\textwidth]{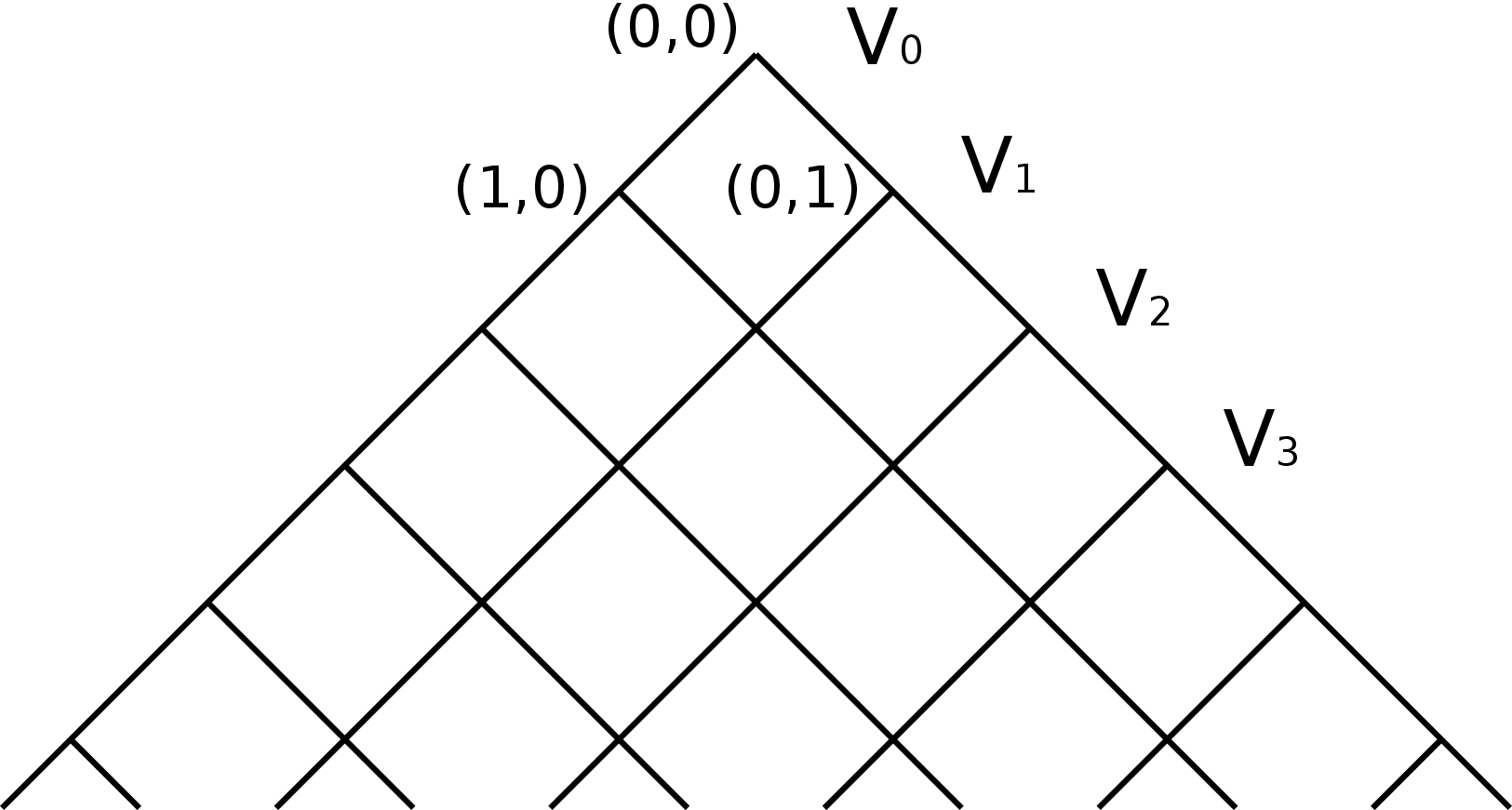}}
\caption{\label{fig:pascal-diagram}Pascal diagram}
\end{figure}

The vertices of the \textit{Pascal diagram} can be labeled by pairs of non-negative integers $(i,j)$. Each vertex $(i,j)$ is connected by an edge to two vertices $(i+1,j)$ and $(i,j+1)$. Typically, this diagram is drawn expanding downwards with the vertex $O=(0,0)$ at the top (see Figure \ref{fig:pascal-diagram}). The vertices are subdivided into levels $V_n=\{(i,j)\in \mathbb Z_+\times\mathbb Z_+:i+j=n\}$. Observe that each vertex $v\in V_n$ is connected by an edge with one or two vertices from $V_{n-1}$ for $n\geqslant 1$. Geometrically, it will be useful for us to view the Pascal diagram (triangle) as a subset of $\mathbb R_+\times\mathbb R_+\subset \mathbb R\times\mathbb R$. By a path in the Pascal diagram, we will mean a (finite or infinite) sequence of edges 
$\{e_0,e_{1},e_{2},\ldots\}$ of the diagram such that $e_j$ connects a vertex of $V_j$ with a vertex of $V_{j+1}$ and is connected to $e_{j+1}$ for each $j$.

Ergodic tail-invariant measures on the Pascal diagram are labeled by a real number (probability) $0 < p < 1$ such that each edge of the form $(v,v+(1,0))$ is given weight $p$ and each edge of the form $(v,v+(0,1))$ is given weight $1-p$ (see e.g. \cite{MelaPetersen2005}). Denote the corresponding measure on the path space of the Pascal diagram by $\nu_p$. We can formally include the values
$p = 0$ and $p = 1$ in the consideration, but this case gives atomic measures $\nu_p$ and is not interesting.


\begin{thm}\label{thm: measure supports}
For each $0 < p < 1$ and each $\varepsilon > 0$ there exists a subdiagram $B_p$  of the Pascal diagram such that 
\begin{enumerate} \item[$1)$] 
$\nu_p(X_p) > 1 - \varepsilon$ for each $p$, where $X_p$ is the path space of $B_p$; 
\item[$2)$] the subspaces $X_p,0\leqslant p\leqslant 1$, are pairwise disjoint.
\end{enumerate}
\end{thm}
\begin{proof} 

Fix $0<p<1$. We will construct by induction a sequence $N_i=N_i(p)$ such that for the sets 
$$A_i=A_i(p)=\left\{\{(x_n,y_n)\}\in X_B:\left|\frac{x_k}{k} -p\right|<2^{1-i}\;\;\text{for all}\;\;k > N_i\right\}$$ 
one has  $\nu_p(\bigcap_{i\leqslant j}A_i)>1 - \varepsilon$ for every $j\in\mathbb N$. Let $N_0=0$.
\vskip 0.1cm
\noindent {\bf Base of induction.} Since $0\leqslant x_n\leqslant n$ for all $n$ we set $A_0=X_B$. One has $\nu_p(A_0)=1.$
\vskip 0.1cm
\noindent {\bf Step of induction.} Let $j\in\mathbb N$. Assume that $N_i$ are constructed for $0\leqslant i\leqslant j$ such that 
$$
\nu_p\left(\bigcap_{i\leqslant j}A_i\right) > 1 - \varepsilon.
$$ 
By the Law of Large Numbers, for $\nu_p$ almost all $\{(x_n,y_n)\}\in X_B$ one has 
$$
\lim_{n \rightarrow \infty} \frac{x_n}{n} = p.
$$ 
It follows that for $N_{j+1}$ large enough one can make $\nu_p(A_{j+1})$ be arbitrary close to $1$. This implies that for $N_{j+1}$ large enough we have
$$\nu_p\left(\bigcap_{i\leqslant j+1}A_i\right)
> 1 - \varepsilon,
$$
and that proves the induction step.

Now, let $A(p)=\bigcap_{i\in\mathbb Z_+}A_i(p)$. Introduce a subdiagram $B_p=(V_p,E_p)$ of the Pascal diagram as follows. Let $V_p$ be the set of all vertices $v\in \mathbb Z_+\times\mathbb Z_+$ such that there exists a path in $A(p)$ containing $v$. Let $E_p$ be the set of all edges of the Pascal diagram having both ends in $V_p$. 

Condition $1)$ of Theorem \ref{thm: measure supports} is satisfied by construction. Let $0 < p\neq q < 1$. Then there exists $i$ such that $|p-q|>2^{2-i}$. By definition, $A_i(p)$ and $A_i(q)$ are disjoint. This implies that condition $2)$ is satisfied as well.
\end{proof}

\begin{remark}
    The subdiagrams $B_p$ are not pairwise disjoint as graphs, but their path spaces $X_p$ are disjoint.
\end{remark}

\section{Orders on the Pascal-Bratteli diagram}
\label{Sect:X_max_X_min}


The following theorem shows that there is an order on the Pascal-Bratteli diagram for which the sets of minimal infinite and maximal infinite paths are of the cardinality continuum. This result can also be found in \cite[Example 7.2]{Fricketal2017},  where the authors provide a sketch of its proof. Here we provide a detailed proof, which has a similar idea to the proof in \cite{Fricketal2017}, but the construction has principal differences with the construction from \cite{Fricketal2017}. The main idea of our proof is constructing minimal paths along every direction. This is motivated in part by Theorem \ref{thm: measure supports}, where every subdiagram is constructed along a particular direction of the Pascal-Bratteli diagram. We hope that this second geometrical approach will contribute to the further development of the theory.

\begin{thm}\label{thm: continuum Xmin} There exists an ordering of the edges of the Pascal diagram into $0,1$ such that both the set of minimal paths $X_{\min}$ and the set of maximal paths $X_{\max}$ have the cardinality of the continuum.
\end{thm}

\begin{proof}
Let $\mathbb D$ be the subset of dyadic numbers of $[0,1)$, \ie \;
$\mathbb D=\{\frac{p}{2^n}:0\leqslant p\leqslant 2^n-1,\;p,n\in\mathbb Z_+\}$. We will construct inductively a countable collection of infinite paths $C_r$, $r\in \mathbb D$, which start at vertices $v_r$. 
We set $C_0$ to be the ``left side of the Pascal triangle'', i.e., the sequence of edges $((j,0),(j+1,0)),j\in\mathbb Z_+$. For $r\in[0,1]$ we denote by $L_r$ the ray inside $\mathbb R_+\times\mathbb R_+$ starting at $(0,0)$ and making the angle $r\frac{\pi}{2}$ with the real positive 
half-line. Thus, $L_0$ and $L_1$ are the sides of the Pascal triangle, and, as a set of points, $C_0$ coincides with $L_0$.

On the $n$th step, $n\in\mathbb Z_+$, given that $C_r$ are constructed for all $r$ of the form $\frac{p}{2^n}$, where $0\leqslant p\leqslant 2^n-1$, we construct the paths $C_r$ for all $r\in\mathbb D$ of the form $\frac{p}{2^{n+1}}$, where $p$ is odd. Moreover, we will do it such that the following inductive conditions hold for each of the constructed paths $C_r$: 
\begin{itemize}
    \item[$1)$] any two paths $C_{r_1}, C_{r_2}$, $r_1\neq r_2$, intersect at most at one point, which coincides with the beginning of one of them;
    \item[$2)$] for any $n,p\in\mathbb Z_+,0\leqslant p\leqslant 2^n-1$ and $r=p/2^n$ the path $C_{r+1/2^{n+1}}$ starts at a vertex $v_{r+1/2^{n+1}}$ of the path $C_r$;
    \item[$3)$] there exists $N=N(r)$ such that for all $(x,y)\in C_r$ with $x+y\geqslant N$ one has \\ $d((x,y),L_r)<2$.
    \end{itemize}

\begin{figure}[ht]
\centerline{\includegraphics[width=0.7\textwidth]{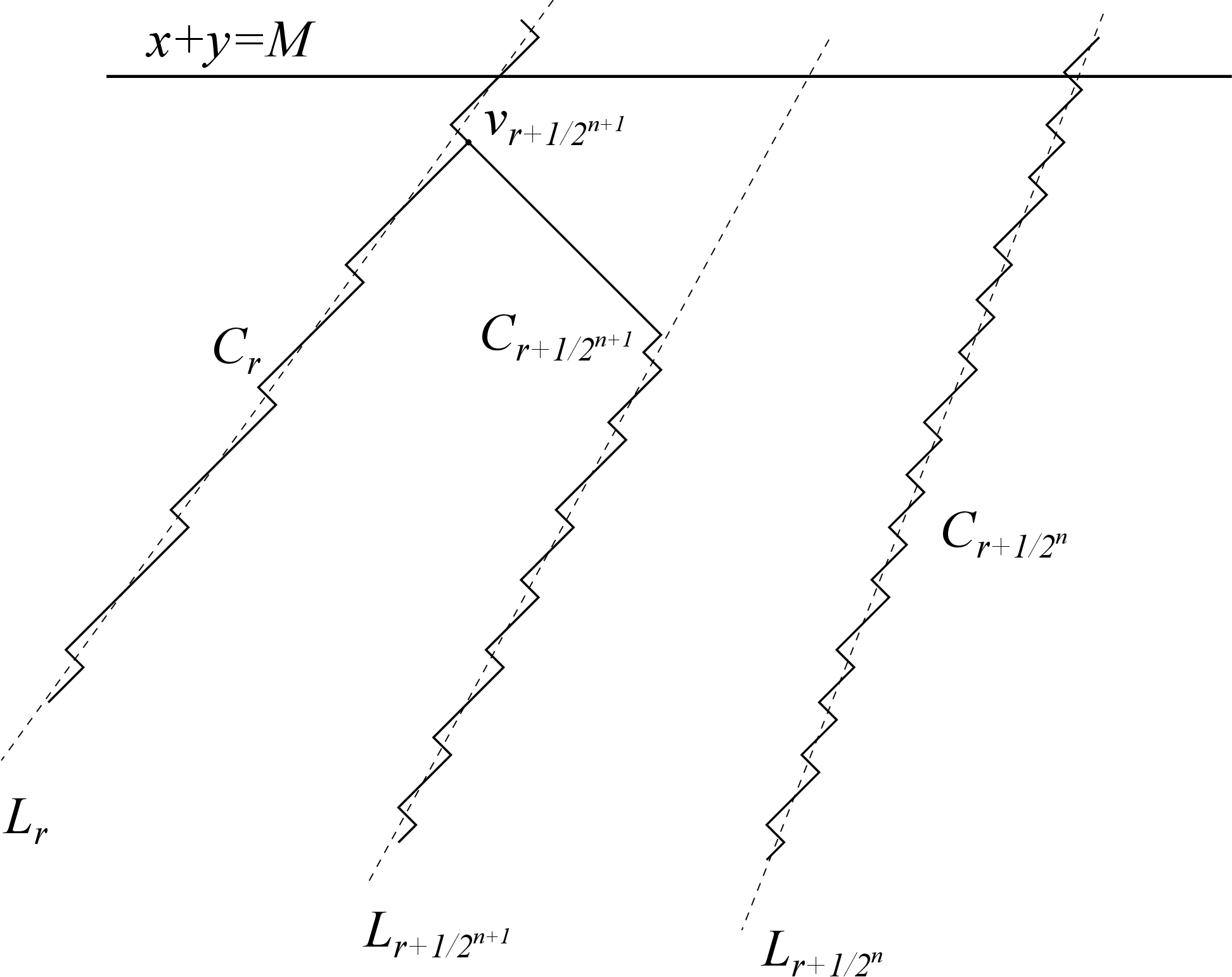}}
\caption{\label{fig:continuum-paths}Illustration to the proof of Theorem~\ref{thm: continuum Xmin}}
\end{figure}

For $n\in\mathbb N$, set $\mathbb D_n=\{\frac{p}{2^n}:p\in\mathbb Z_+,0\leqslant p\leqslant 2^n-1\}$. Assume that $C_r$ are constructed for all $r\in\mathbb D_n$ such that the inductive conditions $1) - 3)$ hold. Set
$$
M=\max\{\max\{N(r):r\in\mathbb D_n\},2^{n+5}\}.$$  Let $H_M$ be the half-plane $\{(x,y):x+y\geqslant M\}$. 

Now, fix $r\in\mathbb D_n$. We will show how to construct the path $C_{r+1/2^{n+1}}$. The inductive condition implies that \begin{equation}\label{eq: dCr} d(C_t\cap H_M,L_{r+1/2^{n+1}}\cap H_M)\geqslant 5,\;\;\text{for all}\;\;t\in\mathbb D_n.\end{equation} Indeed, by $3)$, $C_t\cap H_M$ lies in the neighborhood of radius $2$ of $L_t$. Since $M\geqslant 2^{n+5}$ for any two points $A\in L_t\cap H_M,B\in L_{r+1/2^{n+1}}\cap H_M$ one has 
$$AB\geqslant OA\cdot \sin(1/2^{n+1})>2^{n+4}/2^{n+1}=8.
$$ 
Thus, $d(L_t\cap H_M,L_{r+1/2^{n+1}})>8$. From the above equation \eqref{eq: dCr} follows.

 Let $(a,b)$ be a vertex of the part of the path $C_r$ belonging to $H_M$ minimizing the distance from $C_r\cap H_M$ to $L_{r+1/2^{n+1}}\cap H_M$. Starting with $w_0=(a,b)$ we construct a sequence $\{w_n\}$ of vertices of consecutive levels (i.e. $w_n\in V_{a+b+n}$) such that $w_{n+1}$ is the closest to $L_{r+1/2^{n+1}}$ vertex among two downward neighbors of $w_n$ (belonging to $V_{a+b+n+1}$). We claim that the infinite path $C_{r+1/{2^{n+1}}}$ starting with $w_0$ and passing through the sequence of vertices $\{w_n\}$ satisfies the inductive assumptions.

 Indeed, by the inductive assumptions $1)$ and $2)$ the graph $\Gamma_n$ consisting of the union of the paths $C_t$, $t\in\mathbb D_n$, is a connected planar graph forming an infinite tree. By \eqref{eq: dCr}, $C_t$  does not intersect $L_{r+1/2^{n+1}}\cap H_M$ for all $t\in\mathbb D_n$. Observe that $C_t$ intersects $\partial H_M$ exactly at one point $A_t$ for each $t\in \mathbb D_n$.  Using condition $3)$ and the choice of $M$ we obtain that $C_r\cap H_M$, the segment $A_rA_{r+1/2^n}\subset\partial H_M$, and $C_{r+1/2^n}\cap H_M$ bound a connected infinite domain $U_r$ which contains $H_M\cap L_{r+1/2^{n+1}}$. 

 Consider the ray $w_0+\mathbb R_+=\{w_0+(0,s):s\in\mathbb R_+\}$. Simple geometric considerations show that $w_0+\mathbb R_+$ intersects $L_{r+1/2^{n+1}}$ at some point $P$. By the choice of $w_0$, the ray $w_0+\mathbb R_+$ intersects $C_r$ only at the point $w_0$. 
 Let $l=[|w_0P|]$ (the integer part of the length of the segment joining $w_0$ and $P$). By construction, the points $w_1,w_2,\ldots,w_l$ are consecutive integer points on $w_0+\mathbb R_+$. One has $d(w_l,C_{r+1/2^{n+1}})<1$. Recall that for every $n\geqslant l$ the point $w_{n+1}$ is the closest to $L_{r+1/2^{n+1}}$ among the two neighbors of $w_n$ in $V_{a+b+n+1}$. Using induction we obtain that $d(w_n,C_{r+1/2^{n+1}})<1$ for every $n\geqslant l$. Using \eqref{eq: dCr} we derive that the path $C_{r+1/2^{n+1}}$ lies inside the domain $U_r$, except for the point $w_0$ belonging to $C_r$. From the above considerations, the conditions $1)-3)$ are satisfied for the path $C_{r+1/2^{n+1}}$.

 Furthermore, the union 
 $$
 \Gamma=\bigcup\limits_{n\in\mathbb N} \Gamma_n=\bigcup\limits_{r\in\mathbb D} C_r
 $$ 
 as a graph is a planar infinite tree such that there are no two vertices $v\neq w$ of the same level $V_n$ which are connected by an edge of $\Gamma$ to the same vertex $u\in V_{n+1}$. Note that $\Gamma$ is a subgraph of the Pascal diagram. Set $\Gamma'=\bigcup\limits_{r\in\mathbb D/4} C_r$ and let $\Gamma''$ be the graph symmetric to $\Gamma'$ with respect to $y=x$. Since $C_{1/4}$ asymptotically lies in $2$-neighborhood of the line $y=\sin(\pi/8)x$, there exists $N$ such that $(N,0)+C_{1/4}$ does not intersect the line $y=x$. Then the shifted graphs $(N,0)+\Gamma'$ and $(0,N)+\Gamma''$ do not intersect. Assign $0$ to each edge of $(N,0)+\Gamma'$ and $1$ to each edge of $(0,N)+\Gamma''$ (except the edges on the $y$-axis, which are labeled $0$ automatically). The rest of the edges of the Pascal diagram are numbered by $ 0$'s and $ 1$'s consistently in an arbitrary way. 

 Finally, in addition to already constructed paths $C_r,r\in\mathbb D$, for each $r\in [0,1/4)\setminus \mathbb D/4$ we construct a path $C_r$ in $\Gamma$ (thus, giving a minimal path and a symmetric maximal path in the ordered Pascal diagram) as follows. Write the dyadic expansion of $r$ and the corresponding partial sums: $$r=\sum\limits_{i=1}^\infty 2^{-n_i},\;\;r_k=\sum\limits_{i=1}^k 2^{-n_i},\;\;\text{where}\;\;n_i\in\mathbb N,\;\;n_{i+1}>n_i\;\;\text{for each}\;\;i.$$ Set $r_0=n_0=0$. For $k\in\mathbb Z_+$ denote by $T_k=T_k(r)$ the part of the path $C_{r_{n_k}}$ joining the vertices $v_{r_{n_k}}$ and $v_{r_{n_{k+1}}}$. Set $C_r=\bigcup_{k\in\mathbb Z_+}T_k$. 
 
 By construction, $((0,0),(N,0))\cup ((N,0)+C_r),r\in [0,1/4)$ are pairwise distinct minimal paths of the ordered Pascal diagram. 
 The maximal paths in the diagram are taken symmetrically to the minimal ones with respect to $y=x$. This finishes the proof.
\end{proof}

The following result answers Question 7.3 from \cite{Fricketal2017}.
\begin{thm}\label{thm: continuum Xmin countable Xmax} There exists an ordering of the edges of the Pascal diagram into $0,1$ such that the set of minimal paths $X_{\min}$ has the cardinality of continuum and the set of maximal paths $X_{\max}$ is countably infinite.
\end{thm}
\begin{proof}
    We will construct in a certain way a tree of minimal paths. For each $k\in\mathbb N$ let $S_k\subset V_{4^k}$ be the set of vertices given by:
    $$S_k=\{(2j,4^k-2j):1\leqslant j\leqslant 2^k\}.$$ For each $k$, we connect the vertices from $S_k$ by a collection of paths $\Upsilon_k$ to the vertices from $S_{k+1}$ as follows. For each $1\leqslant j\leqslant 2^k$ connect $(2j,4^k-2j)\in S_k$ by a straight path to $(2j,4^{k+1}-4j)$, the latter connect by a straight path to $(4j-2,4^{k+1}-4j)$, and the latter connect by segments of length 2 to $(4j,4^{k+1}-4j)\in S_{k+1}$ and $(4j-2,4^{k+1}-4j+2)\in S_{k+1}$. Also, let $\Upsilon_0$ be the union of the segment $((0,0),(4,0))$ and the segment $((2,0),(2,2))$. Let $\Upsilon=\bigcup\limits_{k\in\mathbb Z_+}\Upsilon_k$. Notice that $\Upsilon$ is an infinite tree such that for any vertex $(i,j)\in\mathbb N\times\mathbb N$ it contains at most one of the edges $((i-1,j),(i,j))$ and $((i,j-1),(i,j))$. Mark all the edges belonging to $\Upsilon$ with 0. Regardless of how we mark the rest of the edges, $\Upsilon$ contains a continuum of minimal paths.

    Next, let $\Psi$ be the set of edges of the form $((i,j),(i,j+1))$,  not belonging to $\Upsilon$, where $(i,j)\in \mathbb N\times\mathbb Z_+$ is such that $$((i-1,j+1),(i,j+1))$$ does not belong to $\Upsilon$ either. Mark all edges from $\Psi$ by 0 as well. In addition, all the edges belonging to the coordinate axis (the sides of the Pascal graph) are marked with 0 automatically. Mark all other edges by 1. We claim that such numbering of the edges of the Pascal graph has only countably many maximal paths.

    Indeed, observe that by construction for each $n\in\mathbb Z_+$ the line $y=n$ (containing all vertices of the form $(i,n)$, $i\in\mathbb \mathbb Z_+$)) intersects the union of edges marked by 0 by at most one segment. Namely, when $n=4^k-4j$ for some $k\in\mathbb N$ and $1\leqslant j\leqslant 2^{k-1}$. Since such $n$ is necessarily even (in fact, divisible by four), we obtain by definition of $\Psi$ that every edge of the form $((i,2j),(i,2j+1))$ is marked with 0 for $(i,j)\in\mathbb Z_+\times\mathbb Z_+$.
    
    Let $(i,j)\in\mathbb Z_+\times\mathbb Z_+$, $(i,j)\neq (0,0)$ be a branching point for maximal paths, i.e. both edges $((i,j),(i+1,j))$ and $((i,j),(i,j+1))$ are marked by 1. Since $((i,j),(i,j+1))$ does not belong to $\Psi$, we conclude that either $i=0$ or $((i-1,j+1),(i,j+1))$ belongs to $\Upsilon$. In the latter case, $j+1$ is even. By construction, all the edges of the form $((l,j+1),(l,j+2))$, $l\in \mathbb Z_+$, are marked with 0. Thus, any maximal path passing through $(i,j+1)$ may continue only along the line $y=j+1$ to infinity without any further branching possible. We obtain that any maximal infinite path may contain at most two branching points. Since for any two branching points, only two maximal infinite paths might pass through both of them, there are only countably many maximal infinite paths.
\end{proof}

The following proposition about the measure of minimal and maximal infinite paths is true from general observations for other diagrams (see \cite[Lemma 2.7]{BezuglyiKwiatkowskiMedynetsSolomyak2010}), and was proved in \cite[Proposition 2.1]{Fricketal2017}. To illustrate the ideas, we present here direct calculations for the Pascal-Bratteli diagram. 
\begin{prop}
    Let $B$ be a Pascal-Bratteli diagram. Then for any ordering on $B$ and any non-atomic probability ergodic invariant measure $\mu$, the set of minimal infinite paths $X_{\min}$ and the set of maximal infinite paths $X_{\max}$ have $\mu$-measure zero.
\end{prop}

\begin{proof}
    Let $\mu = \nu_p$, where $p \in (0,1)$ (see Section \ref{Sect:pos_meas_sbd_standard_Pascal}), be any non-atomic ergodic probability invariant measure on $B$ and $\omega$ be any order on $B$. 
    We say that a cylinder set is a minimal cylinder set of length $n$ if it corresponds to a minimal finite path of length $n$. Let $X_{\min}^{(n)}$ be a union of minimal cylinder sets of length $n$. Then 
    $$
    X_{\min}^{(n)} \supset X_{\min}^{(n+1)} \mbox{ and } X_{\min} = \bigcap_{n = 1}^{\infty} X_{\min}^{(n)}.
    $$
    Since for every $n \in \mathbb{N}$ and every vertex $w \in V_n$ there is a unique path which joins $w$ with $V_0$, for $p < \frac{1}{2}$, we have
    $$
    \mu(X_{\min}^{(n)}) = \sum_{k = 0}^{n} p^k(1-p)^{n-k} = (1 - p)^n \sum_{k = 0}^{n}\left(\frac{p}{1-p}\right)^{k} \leq (1 - p)^n \frac{1}{1 - \frac{p}{1-p}} = (1 - p)^n\frac{1-p}{1 - 2p} \rightarrow 0
    $$
    as $n$ tends to infinity. By switching $p$ and $1 - p$  we obtain that $\mu(X_{\min}^{(n)})$ tends to $0$ as $n$ tends to infinity for $p > \frac{1}{2}$.
    We also have
    $$
    \mu(X_{\min}^{(n)}) = \sum_{k = 0}^{n} \frac{1}{2^n} = \frac{n+1}{2^n}, \text{ for } p = \frac{1}{2}.
    $$ 
    Thus, for all $p \in (0,1)$ we have
    $$
    \mu(X_{\min}) = \lim_{n \rightarrow \infty}\mu(X_{\min}^{(n)}) = 0. 
    $$
    Similarly, $\mu(X_{\max}) = 0$.
\end{proof}

\begin{remark}
    In \cite{FrickPetersenShields}, the authors study the so-called polynomial shape Bratteli diagrams, which are generalizations of a Pascal-Bratteli diagram. They show that for such diagrams, the sets of minimal infinite and maximal infinite paths have measure zero for a fully supported ergodic invariant probability measure (see \cite[Proposition 3.6]{FrickPetersenShields}). It is also shown that under some mild conditions, the orbits of the sets of minimal infinite and maximal infinite paths are meager (see \cite[Proposition 2.3, Remark 2.4]{FrickPetersenShields}). 
\end{remark}

\section{Stationary generalized Pascal-Bratteli diagrams}\label{Sect:meas_GBD_Pascal}

In this section, we consider generalized Bratteli diagrams that are formed 
by countably many overlapping Pascal triangles. To the best of our knowledge, these diagrams have not been considered before.
\subsection{One-sided stationary generalized Pascal-Bratteli diagram.}
First, we focus on a one-sided stationary generalized Bratteli diagram $B$ that has infinitely many Pascal-Bratteli diagrams as vertex subdiagrams and the
$\mathbb{N} \times \mathbb{N}$ incidence matrix
$$
F = \begin{pmatrix}
    1 & 0 & 0 & 0 & \ldots\\
    1 & 1 & 0 & 0 & \ldots\\
    0 & 1 & 1 & 0 & \ldots\\
    0 & 0 & 1 & 1 & \ldots\\
    \vdots & \vdots & \vdots & \vdots & \ddots
\end{pmatrix}.
$$ 

\vspace{0.5cm}
\begin{figure}[hbt!]
\unitlength=.5cm
\centerline{\includegraphics[width=0.27\textwidth]{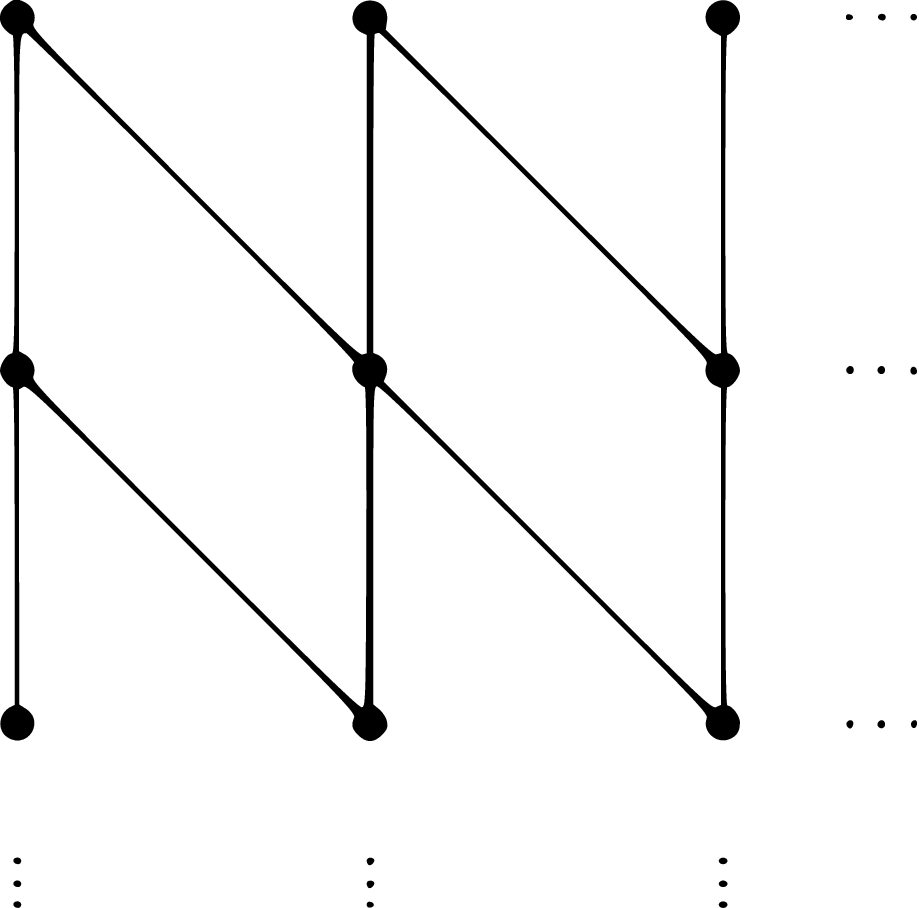}}
\caption{A stationary generalized Bratteli diagram with infinitely many Pascal subdiagrams}\label{Fig:GBD_inf_Pascal_N}
\end{figure}

To be consistent with other papers on generalized Bratteli diagrams, we draw it so that the vertices of consecutive levels are aligned one below another, see Figure \ref{Fig:GBD_inf_Pascal_N}. Diagram $B$ is stationary, and this property allows us to use the techniques from the papers \cite{BezuglyiJorgensenKarpelSanadhya2023}, \cite{BezuglyiKarpelKwiatkowski2024}. For any $i \in \mathbb{N}$, denote by $X_i$ the set of all paths in $X_B$ that start at vertex $i$ on level $V_0$. The sets $\{X_i\}_{i \in \mathbb{N}}$ form a clopen partition of $X_B$. Note that $X_i$ is also the path space of a vertex subdiagram $B_i$ of $B$ that is supported by the set of vertices $W = (W_n)_{n = 0}^{\infty}$, where $W_n = \{i, \ldots, i + n\}$. Obviously, each diagram $B_i$ is the classical Pascal-Bratteli diagram. 
It is well known
that there are uncountably many ergodic probability tail invariant measures $\nu_p^{(i)}$, $0 < p < 1$, on $B_i$ (see Section \ref{Sect:pos_meas_sbd_standard_Pascal}). 
In the proposition below, we describe the invariant measures on $B$ given by eigenpairs for the incidence matrix (see e.g. \cite[Theorem 2.3.2]{BezuglyiJorgensen2022}) or \cite[Theorem 6.6]{BezuglyiJorgensenKarpelSanadhya2023}).

\begin{prop}
    Let $\lambda > 1$ and $\ov \xi_{\lambda} = (\xi_i)$ be the vector such that
    $$
    \xi_i = (\lambda - 1)^{i-1}, \quad i \in \mathbb{N}.
    $$
Then $F^T\ov \xi_{\lambda} = \lambda \ov\xi_{\lambda}$ and the eigenpair $(\ov \xi_{\lambda}, \lambda)$ defines a tail invariant measure $\mu_{\lambda}$. If $1 < \lambda < 2$ then $\mu_{\lambda}$ is finite. 
The measure $\mu_{\lambda}$ is infinite for $\lambda \geq 2$.
\end{prop}

\begin{proof}
    First, we find a non-negative vector $\ov \xi = (\xi_i)$ satisfying $F^T \ov \xi = \lambda \ov \xi$. It is easy to see that if $\xi_1 = 0$ then $\xi_i = 0$ for all $i \in \mathbb{N}$. Set $\xi_1 = 1$. Then $\xi_1 + \xi_2 = \lambda \xi_1$ and $\xi_2 = \lambda - 1$. Similarly, $\xi_3 = \lambda(\lambda - 1) - (\lambda - 1) = (\lambda - 1)^2$ and one can prove by induction that for all $i \in \mathbb{N}$,
    $$
    \xi_i = (\lambda - 1)^{i - 1}.
    $$
    Hence $\ov \xi = \ov \xi_{\lambda}$ is the eigenvector corresponding to $\lambda$. Set
    $$
    \mu_{\lambda}[\ov e^{(n)}(i)] = \frac{\xi_i}{\lambda^n} = \frac{(\lambda - 1)^{i-1}}{\lambda^n},
    $$
    where $[\ov e^{(n)}(i)]$ is a cylinder set corresponding to a finite path $\ov e^{(n)}(i)$ which ends at vertex $i \in V_n$. Clearly, if $1 < \lambda < 2$ then $\mu_{\lambda}$ is finite and
    $$
    \mu_{\lambda}(X_B) = \sum_{i = 0}^{\infty} (\lambda - 1)^i = \frac{1}{2 - \lambda}.
    $$
    If $\lambda \geq 2$ then $\mu_{\lambda}$ is infinite.
\end{proof}


\begin{lemma}
Let $i \in \mathbb{N}$ and $\lambda > 1$. Then the measure $\mu_{\lambda} |_{X_{B_i}}$ (after normalization) coincides with $\nu_p^{(i)}$.
\end{lemma}

\begin{proof}
    Let $m_{\lambda}^{(i)}$ be the normalized measure $\mu_{\lambda} |_{X_{B_i}}$, that is
    $$
    m_{\lambda}^{(i)} = \frac{1}{(\lambda - 1)^i}\mu_{\lambda} |_{X_{B_i}}.
    $$
    We claim that $m_{\lambda}^{(i)} = \nu_p^{(i)}$ where $p = \frac{1}{\lambda}$. Indeed, let $\ov e^{(n)}(j)$ be a finite path from $i$ to $j \in V_n$. Then
    $$
    m_{\lambda}^{(i)}([\ov e^{(n)}(j)]) = \frac{(\lambda - 1)^{j - 1}}{(\lambda - 1)^{i-1}\lambda^n} = \frac{1}{\lambda^n}(\lambda - 1)^{j - i} = \frac{1}{\lambda^{n - j + i}}\left(\frac{\lambda - 1}{\lambda}\right)^{j - i}.
    $$
    Let $j - i = k$. Then 
    $$
    m_{\lambda}^{(i)}([\ov e^{(n)}(j)]) = p^{n-k}q^k,
    $$
    where 
    $$
    p = \frac{1}{\lambda}, \ \ q = \frac{\lambda - 1}{\lambda}.
    $$
    This means that $m_{\lambda}^{(i)} = \nu_p^{(i)}$ for $p = \frac{1}{\lambda}$.
\end{proof}

In the following theorem, we show that all ergodic invariant probability measures on $B$ can be obtained as extensions of the measures $\nu_p^{(i)}$ (see Section \ref{Sect:prelim}).

\begin{theorem}\label{Thm:all_mu_on_GBD_Pascal}
    For any $i \in \mathbb{N}$, the set of measures $\{\wh \nu_p^{(i)}, p \in (\frac{1}{2},1)\}$ is (after normalization) the set of all ergodic probability tail invariant measures on $B$.
\end{theorem}

\begin{proof} 
Let $\mu$ be any ergodic invariant measure on $B$. Without loss of generality, assume that $\mu$ takes the value $1$ on a cylinder set $X_u =[e^{(0)}(u)]$ formed by infinite paths that start at a vertex $u \in V_0$. Note that each cylinder set of $B$ is a standard Pascal-Bratteli diagram. Hence, the restriction of $\mu$ onto $X_u$ must coincide with one of the measures $\nu_p$. Because the measure $\mu|_{X_u}$ coincides with 
$\nu_p$ on a set of positive measure, and both measures are tail invariant, the extensions of these measures must also 
coincide on the saturation of $X_u$ with respect to the tail equivalence relation. 
Thus, $\mu = \wh \nu_p$. 
\end{proof}


The following result is a corollary of Theorem \ref{thm: continuum Xmin}.

\begin{corol}\label{corol}
    There exists a stationary generalized Bratteli diagram together with a (non-stationary) order such that both the set of minimal paths $X_{\min}$ and the set of maximal paths $X_{\max}$ have the cardinality of the continuum.
\end{corol}

\begin{proof}
    By Theorem~\ref{thm: continuum Xmin}, there exists an ordering of a (standard) Pascal-Bratteli diagram such that it has continuum minimal infinite paths and continuum maximal infinite paths. Let $B$ be a stationary generalized Bratteli diagram with infinitely many Pascal subdiagrams. Pick any of its Pascal subdiagrams of level $0$ and enumerate the edges of the subdiagram using Theorem~\ref{thm: continuum Xmin}. All other edges of $B$ can be enumerated in an arbitrary way such that the order is well-defined. Then the obtained order on $B$ has continuum infinite number of both minimal paths and maximal paths.
\end{proof}

The following result is a corollary of Theorem \ref{thm: continuum Xmin countable Xmax}.

\begin{corol}\label{corol_cntbl}
    There exists a stationary generalized Bratteli diagram together with a (non-stationary) order such that the set of minimal paths $X_{\min}$ has the cardinality of the continuum and the set of maximal paths $X_{\max}$ is countably infinite.
\end{corol}

\begin{proof}
   By Theorem~\ref{thm: continuum Xmin countable Xmax}, there exists an ordering of a (standard) Pascal-Bratteli diagram such that it has continuum minimal infinite paths and countably infinitely many maximal infinite paths. Let $B$ be the one-sided stationary generalized Bratteli diagram with infinitely many Pascal subdiagrams. Pick its leftmost Pascal subdiagram, i.e., the Pascal subdiagram which starts at the vertex $1$ of level $0$, and enumerate the edges of the subdiagram using Theorem~\ref{thm: continuum Xmin countable Xmax}. Enumerate all other edges of $B$ from left to right. Then the obtained order on $B$ has continuum minimal infinite paths and infinitely countably many maximal infinite paths.
\end{proof}

Note that for every order on $B$, the leftmost vertical path is simultaneously a minimal infinite path and a maximal infinite path, hence there is no order on $B$ such that the sets $X_{\min}$ and $X_{\max}$ are empty. Even if a generalized Bratteli diagram has a unique minimal and a unique maximal path, it does not necessarily follow that the corresponding Vershik map is a homeomorphism \cite[Theorem 4.7]{BezuglyiJorgensenKarpelSanadhya2023}. In the following we consider a two-sided stationary generalized Pascal-Bratteli diagram $\tl B$ and show in Proposition \ref{prop: no min max paths} that there is an order on $\tl B$ such that $X_{\min} = X_{\max} = \emptyset$, and hence the corresponding Vershik map is a homeomorphism.

\subsection{Two-sided infinite stationary generalized Pascal-Bratteli diagram.}
We can also consider a two-sided generalized Bratteli diagram 
$\tl B$ which has infinitely many Pascal-Bratteli diagrams as subdiagrams.
Diagram $\tl B$ has a $\mathbb{Z} \times \mathbb{Z}$ incidence matrix (we use the boldface to indicate 
the entries on the main diagonal):
$$
\tl F = \begin{pmatrix}
   \ddots & \vdots 
   & \vdots & \vdots & \vdots & \vdots & \udots\\
    \cdots & \textbf 1 & 0 &  0 & 0 & 0 & \cdots\\
    \cdots & 1 & \textbf 1 &  0 & 0 & 0 & \cdots\\
    \cdots & 0 &  1 & \textbf{1} &  0 &   0 & \cdots\\
    \cdots & 0 & 0 &  1 & \textbf 1 & 0 & \cdots\\
    \cdots & 0 & 0 & 0 & 1 & \textbf 1 & \cdots\\
    \udots & \vdots & \vdots & \vdots & \vdots & \vdots & \ddots\\
\end{pmatrix}.
$$ 
Diagram $\tl B$ is both vertically and horizontally stationary (see \cite{BJKK_HS_2024}). Corollaries \ref{corol} and \ref{corol_cntbl} also hold for the two-sided generalized Bratteli diagram $\tl B$. Indeed, to prove Corollary \ref{corol_cntbl}, it is enough to pick any Pascal subdiagram $B'$ of level 0 and enumerate the edges of the subdiagram using Theorem~\ref{thm: continuum Xmin countable Xmax}. All other edges of $\tl B$ which are to the right of $B'$ should be enumerated from left to right, and all edges to the left of $B'$ should be enumerated from right to left.

In the following proposition, we show that there is an order on 
$\tl B$ does not admit any minimal or maximal infinite paths.

\begin{prop}\label{prop: no min max paths} There exists an ordering on $\tl B$ that does not have minimal infinite or maximal infinite paths.
\end{prop}
\begin{figure}[ht]
\centerline{\includegraphics[width=0.5\textwidth]{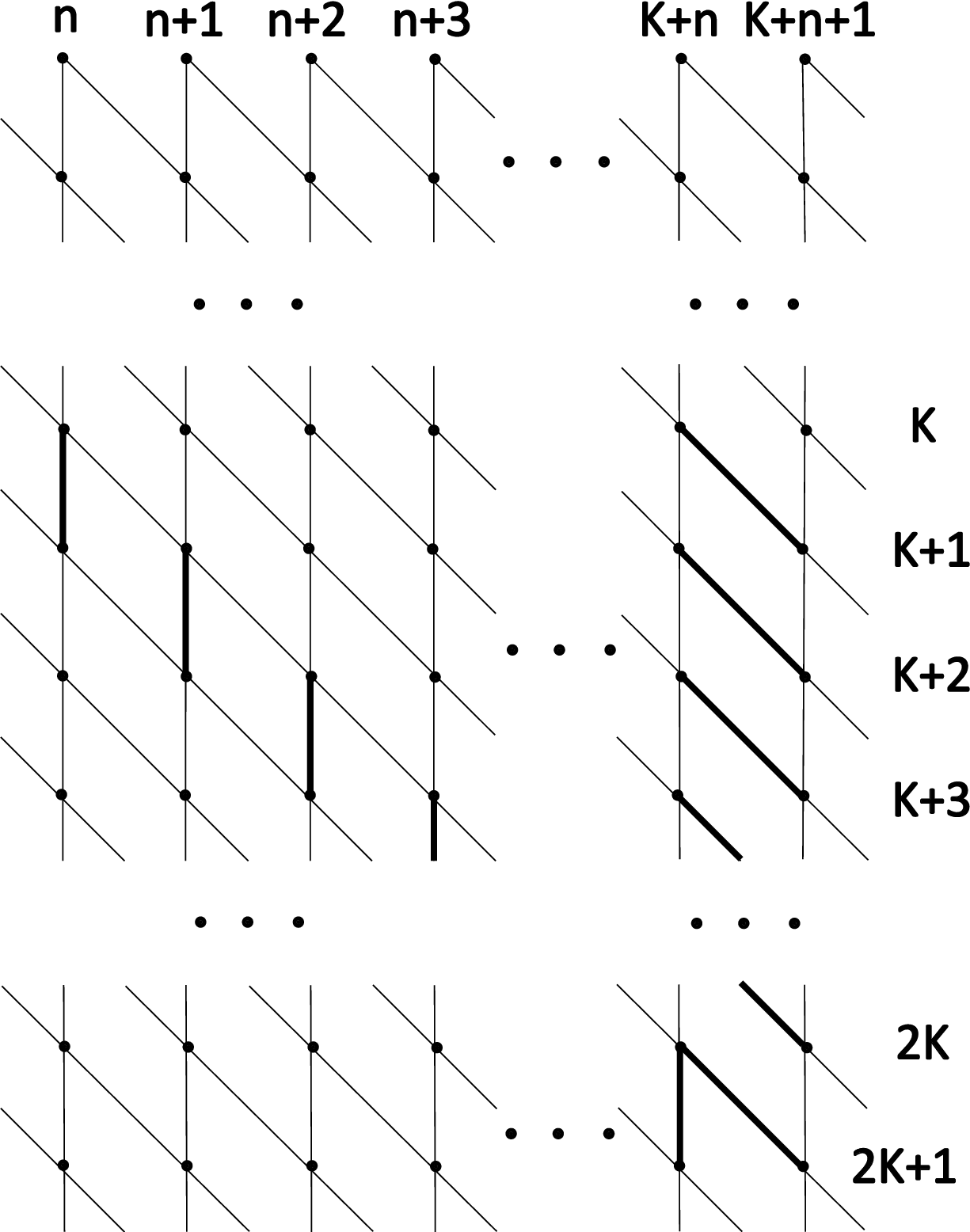}}
\caption{\label{fig:no-min-max-paths}Illustration to the proof of Proposition~\ref{prop: no min max paths}: edges from $\mathcal S(n,i)$ are in bold}
\end{figure}
\begin{proof} Fix an injective map $G$ from $\mathbb Z\times \{0,1\}$ to $\mathbb N$. For each $(n,i)\in\mathbb Z\times\{0,1\}$ set  $K=K(n,i)=3^{G(n,i)}$. For $n\in\mathbb Z,\ k\in\mathbb Z_+$, let $n^{(k)}$ be the $n$-th vertex of the $k$-th level of the generalized Bratteli diagram $B$.  Let $\mathcal S(n,i)$ be the set of all the edges of the form $((l+n)^{(K+l)},(l+n)^{(K+l+1)})$
and of the form $((K+n)^{(K+l)},(K+n+1)^{(K+l+1)})$,  
 where $0\leqslant l\leqslant  K=K(n,i)$. Mark all the edges from $\mathcal S(n,i)$ with $i$. It is not hard to see that the sets $\mathcal S(n,i)$, $(n,i)\in\mathbb Z\times \{0,1\}$, are pairwise disjoint. Therefore, the above operation is well-defined. Number all the other edges in an arbitrary consistent way to obtain an ordering on $\tl B$.

For any top  vertex $n^{(0)}$, $n\in\mathbb N$, of the diagram $\tl B$ and any $i\in\{0,1\}$ any infinite path starting at $n^{(0)}$ passes through an edge of $\mathcal S(n,i)$. It follows that there are no minimal or maximal paths.
\end{proof}

For the two-sided generalized Bratteli diagram $\tl B$, for every $p \in (0,1)$, there is a positive right eigenvector (the boldface indicates the zero coordinate of the eigenvector):
$$
\mathbf{x}^T = \left(\ldots, \; \frac{(1-p)^2}{p^2},\; \frac{1-p}{p},\; \textbf{1},\; \frac{p}{1-p},\; \frac{p^2}{(1-p)^2}, \; \ldots\right) 
$$
of $\tl A = \tl F^T$ which corresponds to eigenvalue $\lambda = \frac{1}{p}$.
In other words, we have $\mathbf{x} = (x_i)_{i \in \mathbb{Z}}$ and $x_i = \frac{p^i}{(1-p)^i}$ for $i \in \mathbb{Z}$.
Indeed, for every $i \in \mathbb{Z}$ we have
$$
(\tl A\mathbf{x})_i = \left(\frac{p}{1-p}\right)^{i-1} + \left(\frac{p}{1-p}\right)^{i} = \frac{1}{p}\left(\frac{p}{1-p}\right)^{i} = \lambda x_i.
$$
In particular, for $p = \frac{1}{2}$, we obtain $\lambda = 2$ and $
\mathbf{x}^T = \left(\ldots, \; 1,\; 1,\; 1,\; \ldots\right). 
$
Note that for every $p \in (0,1)$, the corresponding eigenvector $\mathbf{x}$ has an infinite sum of coordinates and hence generates an infinite $\sigma$-finite measure.

\medskip
\textbf{Acknowledgements.} We extend our gratitude to K. Petersen for pointing out the connections and similarities between some results of \cite{Fricketal2017} and \cite{FrickPetersenShields}, and our results proved in Section \ref{Sect:X_max_X_min}. We are grateful to the reviewers for careful reading of the paper and valuable suggestions.
We are also thankful to our colleagues, 
especially, P. Jorgensen, J. Kwiatkowski, C. Medynets, T. Raszeja, and S. Sanadhya for the numerous valuable discussions.
S.B. and O.K. are grateful to the Institute of Mathematics of the Polish Academy of Sciences for their hospitality and support. 
O.K. is supported by the NCN (National Science Centre, Poland) Grant 2019/35/D/ST1/01375 and by the program ``Excellence initiative - research university'' for the AGH University of Krakow. A.D. acknowledges the funding by the Long-term program of support of the Ukrainian research teams at the Polish Academy of Sciences carried out in collaboration with the U.S. National Academy of Sciences with the financial support of external partners. A.D. was also partially supported by the National Science Centre, Poland, Grant OPUS21 "Holomorphic dynamics, fractals, thermodynamic formalism", 2021/41/B/ST1/00461.

\bibliographystyle{alpha}
\bibliography{bibliographyPascal}
\end{document}